\documentclass[12pt,reqno]{article}

\usepackage[usenames]{color}
\usepackage{amssymb}
\usepackage{graphicx}
\usepackage{amscd}

\usepackage[colorlinks=true,
linkcolor=webgreen,
filecolor=webbrown,
citecolor=webgreen]{hyperref}

\definecolor{webgreen}{rgb}{0,.5,0}
\definecolor{webbrown}{rgb}{.6,0,0}

\usepackage{color}
\usepackage{fullpage}
\usepackage{float}
\usepackage{graphics,amsmath,amssymb}
\usepackage{amsthm}
\usepackage{amsfonts}
\usepackage{latexsym}

\begin{document}

	\theoremstyle{plain}
	\newtheorem{theorem}{Theorem}
	\newtheorem{lemma}[theorem]{Lemma}
	\newtheorem{corollary}[theorem]{Corollary}
	\theoremstyle{definition}
	\newtheorem{definition}[theorem]{Definition}
	\newtheorem{example}[theorem]{Example}
	\newtheorem{remark}[theorem]{Remark}

\begin{center}
	\vskip 1cm{\LARGE\bf
		Chebyshev-Fibonacci polynomial relations using generating functions\\
		\vskip .1in } \vskip 0.5cm \large
Robert Frontczak\footnote{Statements and conclusions made in this article by R. Frontczak are entirely those of the author. They do not necessarily reflect the views of LBBW.}\\
Landesbank Baden-W{\"u}rttemberg,\\
Stuttgart, Germany\\
\href{mailto:robert.frontczak@lbbw.de}{\tt robert.frontczak@web.de}\\
\vskip .2 in

		Taras Goy\\
	Faculty of Mathematics and Computer Science,\\
	Vasyl Stefanyk Precarpathian National University, \\
	Ivano-Frankivsk, Ukraine\\
	\href{mailto:taras.goy@pnu.edu.ua}{\tt taras.goy@pnu.edu.ua} \\
	\vskip .2 in
\end{center}

\vskip .2 in

\begin{abstract}
The main object of the paper is to reveal connections between Chebyshev polynomials of the first and second kinds and Fibonacci polynomials introduced by Catalan. This is achieved by relating the respective (ordinary and exponential) generating functions to each other. As a consequence, we also establish new combinatorial identities for balancing polynomials and Fibonacci (Lucas) numbers.
\end{abstract}

\section*{1. Introduction}

For any integer $n\geq0$, the Chebyshev polynomials $\{T_n(x)\}_{n\geq0}$ of the first kind are de\-fined by the second-order recurrence relation \cite{Mason}
\begin{equation}\label{T-def}
T_0(x)=1,\quad T_1(x)=x,\quad T_{n+1}(x)=2xT_n(x)-T_{n-1}(x),
\end{equation}
while the Chebyshev polynomials $\{U_n(x)\}_{n\geq0}$ of the second kind are defined by
\begin{equation}\label{U-def}
U_0(x)=1,\quad U_1(x)=2x,\quad U_{n+1}(x)=2xU_n(x)-U_{n-1}(x).
\end{equation}

If we denote $\alpha(x)=x+\sqrt{x^2-1}$ and  $\beta(x)=x-\sqrt{x^2-1}$, then we have   
\begin{gather*}
T_n(x)=\frac{\alpha^n(x)+\beta^n(x)}{2}=\sum_{k=0}^{\lfloor{n}/{2}\rfloor}{n\choose 2k}(x^2-1)^kx^{n-2k},\\
U_n(x)=\frac{\alpha^{n+1}(x)-\beta^{n+1}(x)}{2\sqrt{x^2-1}}=\sum_{k=0}^{\lfloor{n}/{2}\rfloor}{n+1\choose 2k+1}(x^2-1)^kx^{n-2k}.
\end{gather*}

Fibonacci polynomials are polynomials that can be defined by Fibonacci-li\-ke recursion relations. They were studied in 1883 by E. Catalan and E. Jacobsthal. For example, Catalan studied the polynomials $F_n(x)$ defined by the recurrence
\begin{equation*}
F_n(x) = xF_{n-1}(x)+F_{n-2}(x),\quad n\geq2, 
\end{equation*}
with $F_0(x)=0$ and $F_1(x)=1$. A non-recursive expression for $F_n(x)$ is
\begin{equation*}
F_{n}(x)=\frac{\rho^n(x)-\sigma^n(x)}{\sqrt{x^2+4}}=
\sum_{k=0}^{\lfloor\frac{n-1}{2}\rfloor}{n-k-1\choose k}x^{n-2k-1},\quad n\geq0,
\end{equation*}
where $\rho(x)=\frac{x+ \sqrt{x^2+4}}{2}$ and $\sigma(x)=\frac{x-\sqrt{x^2+4}}{2}$.

Chebyshev and Fibonacci polynomials play an important role in applied mathematics. They possess many interesting and unique properties. 
Excellent sources are the textbooks \cite{Koshy,Mason,Rivlin}, among others. 

The most recent literature on Chebyshev and Fibonacci polynomials encompasses the following articles. Kilic~et~al.~\cite{Kilic} computed various types of power sums for Chebyshev polynomials and deduced new connections between Chebyshev polynomials and Fibonacci numbers. 
Kim et al. \cite{Kim}~recently derived new expressions for sums of finite products of Chebychev and Fibonacci polynomials. In \cite{Abd}, Abd-Elhammed~et~al. established new connection formulas between Fibonacci polynomials and Chebyshev polynomials. These formulas are expressed in terms of certain values of hypergeometric functions. Li and Wenpeng \cite{Li-Wenpeng}, using the definitions and properties of Chebyshev polynomials, studied the power sum problems involving Fibonacci polynomials and Lucas polynomials and obtained some interesting divisibility properties. Li \cite{Li}~studied relationships between Chebyshev polynomials, Fibonacci polynomials, and their derivatives, and got the formula for derivatives of Chebyshev polynomials being represented by Chebyshev polynomials and Fibonacci polynomials. Finally, we mention the paper by Fl\'{o}rez, McAnally and Mukherjee \cite{Flo}~where many identities for generalized Fibonacci polynomials 
are derived. 

The purpose of this paper is to obtain some identities involving Chebyshev polynomials of the first and second kinds and Fibonacci polynomials. We achieve this in a conventional manner by relating the respective (ordinary and exponential) generating functions to each other, resulting in a range of interesting functional equations. Our approach is in the spirit of \cite{Frontczak1,Frontczak2,FR-GT}. Also, using simple connections between Chebyshev polynomials and balancing polynomials we will be able to incorporate 
the later polynomial class into our analysis. Some of the results of this paper were announced without proofs in \cite{FR-GT-Vorokhta}. 

\medskip

\section*{2. Some generating functions}

This section contains the generating functions that will be used later in this article. We state the results without proofs 
as they can be derived without much efforts. We recommend the article by Mez\H{o} \cite{Mezo} for a comprehensive study of generating functions for second-order sequences. 

From \eqref{T-def} and \eqref{U-def} it can be shown that the ordinary generating functions 
for Chebyshev polynomials $T_n(x)$, $U_n(x)$ and their odd and even indexed companions are given by 
\begin{gather}
t(z,x)=\sum_{n\geq0}T_{n}(x)z^{n}=\frac{1-xz}{1-2xz+z^2}, \label{gf-t}\\
t_1(z,x)=\sum_{n\geq0}T_{2n+1}(x)z^{n}=\frac{x(1-z)}{1-(4x^2-2)z+z^2}, \label{gf-t1}\\
t_2(z,x)=\sum_{n\geq0}T_{2n}(x)z^{n}=\frac{1-(2x^2-1)z}{1-(4x^2-2)z+z^2}, \label{gf-t2}
\end{gather}
and 
\begin{gather}
u(z,x)=\sum_{n\geq0}U_{n}(x)z^{n}=\frac{1}{1-2xz+z^2}, \label{gf-u}\\
u_1(z,x)=\sum_{n\geq0}U_{2n+1}(x)z^{n}=\frac{2x}{1-(4x^2-2)z+z^2}, \label{gf-u1}\\
u_2(z,x)=\sum_{n\geq0}U_{2n}(x)z^{n}=\frac{1+z}{1-(4x^2-2)z+z^2}. \label{gf-u2}
\end{gather}

In addition, the corresponding exponential generating functions for these polynomial sequences are given by 
\begin{gather}
\tau(z,x)=\sum_{n\geq0}T_n(x)\frac{z^n}{n!}=e^{xz}\cosh\bigl(\sqrt{x^2-1}\,z\bigr), \label{egf-t}\\
\tau_1(z,x)=\sum_{n\geq0}T_{2n+1}(x)\frac{z^n}{n!}\nonumber\\
=e^{(2x^2-1)\,z}\Bigl(x\cosh\bigl(2x\sqrt{x^2-1}\,z\bigr)+\sqrt{x^2-1}\sinh\bigl(2x\sqrt{x^2-1}\,z\bigr)\Bigr), \label{egf-t1}\\
\tau_2(z,x)=\sum_{n\geq0}T_{2n}(x)\frac{z^n}{n!}=e^{(2x^2-1)z}\cosh\bigl(2x\sqrt{x^2-1}\,z\bigr), \label{egf-t2}
\end{gather}
and
\begin{gather}
\omega(z,x)=\sum_{n\geq0}U_n(x)\frac{z^n}{n!}\notag\\
=\frac{e^{xz}}{\sqrt{x^2-1}}\Bigl(x\sinh\bigl(\sqrt{x^2-1}\,z\bigr)+\sqrt{x^2-1}\cosh\bigl(\sqrt{x^2-1}\,z\bigr)\Bigr),\label{egf-u}\\[2pt]
\omega_1(z,x)=\sum_{n\geq0}U_{2n+1}(x)\frac{z^n}{n!}\nonumber\\
=\frac{e^{(2x^2-1)z}}{\sqrt{x^2-1}}\Bigl((2x^2-1)\sinh\bigl(2x\sqrt{x^2-1}\,z\bigr)+2x\sqrt{x^2-1}\cosh\bigl(2x\sqrt{x^2-1}\,z\bigr)\Bigr),\label{egf-u1}
\end{gather}
\begin{gather}
\omega_2(z,x)=\sum_{n\geq0}U_{2n}(x)\frac{z^n}{n!}\notag\\
=\frac{e^{(2x^2-1)z}}{\sqrt{x^2-1}}\Bigl(x\sinh\bigl(2x\sqrt{x^2-1}\,z\bigr)+\sqrt{x^2-1}\cosh\bigl(2x\sqrt{x^2-1}\,z\bigr)\Bigr).\notag 
\end{gather}

Fibonacci polynomials $F_n(x)$, $F_{2n+1}(x)$ and $F_{2n}(x)$ have the following ordinary generating functions
\begin{gather}
f(z,x) =\sum_{n\geq0}F_n(x)z^n=\frac{z}{1-xz-z^2}, \label{gf-f}\\[2pt]
f_1(z,x) =\sum_{n\geq0}F_{2n+1}(x)z^{n}=\frac{1-z}{1-(x^2+2)z+z^2}, \label{gf-f1}\\[2pt]
f_2(z,x) =\sum_{n\geq0}F_{2n}(x)z^{n}=\frac{xz}{1-(x^2+2)z+z^2}, \label{gf-f2}
\end{gather}
while the exponential generating functions are 
\begin{gather}
\label{egf-F}
\phi(z,x)=\sum_{n\geq0}F_n(x)\frac{z^n}{n!} = \frac{2e^{\frac{xz}{2}}}{\sqrt{x^2+4}}\sinh\Bigl(\frac{\sqrt{x^2+4}}{2}\,z\Bigr),\\[2pt]
\phi_1(z,x)=\sum_{n\geq0}F_{2n+1}(x)\frac{z^{n}}{n!}\notag\\
=\frac{e^{\frac{x^2+2}{2}z}}{\sqrt{x^2+4}}\left(x\sinh\Bigl(\frac{x\sqrt{x^2+4}}{2}\,z\Bigr)+\sqrt{x^2+4}\cosh\Bigl(\frac{x\sqrt{x^2+4}}{2}\,z\Bigr)\right),\label{egf-F1}\\[2pt]
\phi_2(z,x)=\sum_{n\geq0}F_{2n}(x)\frac{z^{n}}{n!}
=\frac{2e^{\frac{x^2+2}{2}z}}{\sqrt{x^2+4}}\sinh\Bigl(\frac{x\sqrt{x^2+4}}{2}\,z\Bigr).\label{egf-F2}
\end{gather}

\medskip

\section*{3.~Chebyshev-Fibonacci polynomial identities using \break ordinary generating functions}

In what follows, we will use the standard convention that $\sum_{k=0}^n a_k=0$ for $n<0$.
\begin{theorem} \label{Theo1} 
	For $n\geq1$, the following polynomial identities hold: 
	\begin{gather}
	F_n(x)= T_{n-1}(x)-\sum_{k=1}^{n-2}\bigl(xT_{n-1-k}(x)-2T_{n-2-k}(x)\bigr)F_k(x), \label{f-t}\\
	F_{n}(x)+xF_{n-1}(x)= U_{n-1}(x)-\sum_{k=1}^{n-2}\bigl(xU_{n-1-k}(x)-2U_{n-2-k}(x)\bigr)F_{k}(x). \label{f-u}
	\end{gather}
\end{theorem}
\begin{proof} 
	To prove formula \eqref{f-t}, observe that by \eqref{gf-t} and \eqref{gf-f}, we obtain, respectively,
	\begin{gather*}
	1-xz = \frac{1-xz+t(z,x)(xz-z^2)}{t(z,x)},\qquad 1-xz = \frac{z+z^2f(z,x)}{f(z,x)},
	\end{gather*}
	and thus
	\begin{equation*}
	(1-xz)f(z,x)-zt(z,x) = (2z^2-xz)t(z,x)f(z,x).
	\end{equation*}
	Expanding both sides of the last equation as a power series in $z$ and using the Cauchy product of two power series, we then obtain
	\begin{gather*}
	\sum_{n\geq0}F_n(x)z^n-x\sum_{n\geq0}F_n(x)z^{n+1}-\sum_{n\geq0}T_n(x)z^{n+1}\\
	=2\sum_{n\geq0}\sum_{k=0}^{n}T_{n-k}(x)F_k(x)z^{n+2}-x\sum_{n\geq0}\sum_{k=0}^nT_{n-k}(x)F_k(x)z^{n+1}
	\end{gather*}
	or, equivalently,
	\begin{gather*}
	F_0(x)+F_1(x)z+\sum_{n\geq2}F_n(x)z^n-x\sum_{n\geq2}F_{n-1}(x)z^{n}-T_0(x)z-\sum_{n\geq2}T_{n-1}(x)z^{n}\\
	=2\sum_{n\geq2}\sum_{k=0}^{n-2}T_{n-2-k}(x)F_k(x)z^{n}-x\sum_{n\geq2}\sum_{k=0}^{n-1}T_{n-1-k}(x)F_k(x)z^{n},\\
	\sum_{n\geq2}\bigl(F_n(x)-xF_{n-1}(x)-T_{n-1}(x)\bigr)z^{n}\\
	=\sum_{n\geq2}\Bigl(2\sum_{k=0}^{n-2}T_{n-2-k}(x)F_k(x)-x\sum_{k=0}^{n-1}T_{n-1-k}(x)F_k(x)\Bigr)z^{n}.
	\end{gather*}
	Comparing the coefficients on both sides, we have
	\begin{gather*}
	F_n(x)-xF_{n-1}(x)-T_{n-1}(x)
	= \sum_{k=0}^{n-2}\bigl(2T_{n-2-k}(x)-xT_{n-1-k}(x)\bigr)F_k(x)-xT_0(x)F_{n-1}(x),
	\end{gather*}
	as desired. The proof of \eqref{f-u} is very similar. From \eqref{gf-u} and \eqref{gf-f} the following functional equation follows:
	\begin{equation*}
	\frac{1}{u(z,x)} = \frac{z}{f(z,x)} + z(2z-x)
	\end{equation*}
	or, equivalently,
	\begin{equation*}
	f(z,x) - z u(z,x) = 2 z^2 f(z,x)u(z,x) - x z f(z,x) u(z,x).
	\end{equation*}
	The remainder of the proof is the same as above.
\end{proof}

In a similar manner, we use the generating functions  \eqref{gf-t1}, \eqref{gf-u1}, \eqref{gf-f1}, and \eqref{gf-t2}, \eqref{gf-u2}, \eqref{gf-f2}, respectively, to prove four additional relations between odd (even) indexed Chebyshev and Fibonacci polynomials. These relations are contained in the next theorem, those proofs we leave to the reader.
\begin{theorem}\label{Theo2}
	The following identities hold for $n\geq1$
	\begin{gather*} 
	x\bigl(F_{2n+1}(x)-F_{2n-1}(x)\bigr) = T_{2n+1}(x)-T_{2n-1}(x)-(3x^2-4)\sum_{k=0}^{n-1}F_{2k+1}(x)T_{2(n-k)-1}(x), \\
	2xF_{2n+1}(x) = U_{2n+1}(x)-U_{2n-1}(x)-(3x^2-4)\sum_{k=0}^{n-1}F_{2k+1}(x)U_{2(n-k)-1}(x).
	\end{gather*}
	The even indexed counterparts are given by  
	\begin{gather*}
	F_{2n}(x)-(2x^2-1)F_{2n-2}(x)	= xT_{2n-2}(x)-(3x^2-4)\sum_{k=1}^{n-1}F_{2k}(x)T_{2(n-k-1)}(x),\\
	F_{2n}(x)+F_{2n-2}(x) = xU_{2n-2}(x)-(3x^2-4)\sum_{k=1}^{n-1}F_{2k}(x)U_{2(n-k-1)}(x).
	\end{gather*}
\end{theorem}

Next, we present a range of Chebyshev-Fibonacci identities with mixed indices.
\begin{theorem} 
	For $n\geq1$, we have \label{Theo3}
	\begin{gather}
	xF_{n}(x)+(4x^3-x^2-3x)F_{n-1}(x)\nonumber\\
	=T_{2n-1}(x)-\sum_{k=1}^{n-2}\bigl((4x^2-x-2)T_{2(n-k)-1}(x)-2T_{2(n-k)-3}(x)\bigr)F_{k}(x),\nonumber\\
	2xF_{n}(x)+(8x^3-2x^2-4x)F_{n-1}(x)\nonumber\\
	= U_{2n-1}(x)-\sum_{k=1}^{n-2}\bigl((4x^2-x-2)U_{2(n-k)-1}(x)-2U_{2(n-k)-3}(x)\bigr)F_{k}(x),\nonumber\\
	F_{n}(x)+(2x^2-x-1)F_{n-1}(x)\nonumber\\
	=T_{2n-2}(x)-\sum_{k=1}^{n-2}\bigl((4x^2-x-2)T_{2(n-k-1)}(x)-2T_{2(n-k-2)}(x)\bigr)F_{k}(x),\nonumber\\
	F_{n}(x)+(4x^2-x-1)F_{n-1}(x)\nonumber\\
	=U_{2n-2}(x)-\sum_{k=1}^{n-2}\bigl((4x^2-x-2)U_{2(n-k-1)}(x)-2U_{2(n-k-2)}(x)\bigr)F_{k}(x),\nonumber\\
			F_{2n+1}(x)-(2x^2-1)F_{2n-1}(x)\nonumber\\
	=T_{2n}(x)-T_{2n-2}(x)-(3x^2-4)\sum_{k=0}^{n-1}T_{2(n-k-1)}(x)F_{2k+1}(x),\nonumber\\
		F_{2n+1}(x)+F_{2n-1}(x)\notag\\
	=U_{2n}(x)-U_{2n-2}(x)-(3x^2-4)\sum_{k=0}^{n-1}U_{2(n-k-1)}(x)F_{2k+1}(x),\nonumber 
	\end{gather}
	\begin{gather}
	x^2F_{2n-1}(x)=xT_{2n-1}(x)-(3x^2-4)\sum_{k=1}^{n-1}T_{2(n-k)-1}(x)F_{2k}(x),\nonumber\\ \label{f2-u1}
	2xF_{2n}(x)= xU_{2n-1}(x)-(3x^2-4)\sum_{k=1}^{n-1}U_{2(n-k)-1}(x)F_{2k}(x),\\
	F_{2n+1}(x)-xF_{2n-1} = T_n(x)-T_{n-1}(x)+(x^2-2x+2)\sum_{k=0}^{n-1}T_{n-1-k}(x)F_{2k+1}(x),\nonumber\\
	F_{2n+1}(x) = U_n(x)-U_{n-1}(x)+(x^2-2x+2)\sum_{k=0}^{n-1}U_{n-1-k}(x)F_{2k+1}(x),\nonumber\\
	F_{2n}(x)-xF_{2n-2} =xT_{n-1}(x)+(x^2-2x+2)\sum_{k=1}^{n-1}T_{n-1-k}(x)F_{2k}(x),\nonumber\\
	F_{2n}(x) = xU_{n-1}(x)+(x^2-2x+2)\sum_{k=1}^{n-1}U_{n-1-k}(x)F_{2k}(x).\nonumber
	\end{gather}
\end{theorem}
\begin{proof} We will prove only \eqref{f2-u1}, the others can be proved in a similar way.
	The formula is essentially a consequence of the functional equation
	\begin{equation*}
	2xf_2(z,x)=xzu_1(z,x)-(3x^2-4)zu_1(z,x)f_2(z,x),
	\end{equation*}
	which can be derived from \eqref{gf-u1} and \eqref{gf-f2}.
\end{proof}

It is worth to note that our previous results can be used to establish connections between Fibonacci polynomials and balancing 
and Lucas-balancing polynomials, respectively. Recall that balancing polynomials $B_n(x)$ and Lucas-balancing polynomials $C_n(x)$ are
generalizations of balancing and Lucas-balancing numbers. They are defined by the same recurrence \cite{Frontczak0}
$w_n(x)=6xw_{n-1}(x)-w_{n-2}(x)$, $n\geq2,$ but with different initial values $B_{0}(x)=0$, $B_{1}(x)=1$ and $C_{0}(x)=1$, $C_{1}(x)=3x$, respectively. From the definitions \eqref{T-def} and \eqref{U-def}, the following connections are easily derived (see \cite{Frontczak0}) 
\begin{gather}\label{BU-CT}
B_n(x)=U_{n-1}(3x), \quad C_{n}=T_n(3x), \quad n\geq1.
\end{gather}

In view of \eqref{BU-CT} and Theorems \ref{Theo1}-\ref{Theo3} relations between Fibonacci and balancing (Lucas-balancing) polynomials
are obvious. In the next statement we present only a few of them.
\begin{corollary} For $n\geq1$,
	\begin{gather*}
	F_n(3x) = C_{n-1}(x)-\sum_{k=1}^{n-2}\bigl(3xC_{n-k-1}(x)-2C_{n-k-2}(x)\bigr)F_k(3x),\\
	F_n(3x)+3xF_{n-1}(3x) = B_{n}(x)-\sum_{k=1}^{n-2}\bigl(3xB_{n-k}(x)-2B_{n-k-1}(x)\bigr)F_k(3x),\\
	9x^2F_{2n}(3x) = C_{2n+1}(x)-C_{2n-1}(x)-(27x^2-4)\sum_{k=1}^{n-1}C_{2(n-k)-1}(x)F_{2k+1}(3x),
		\end{gather*}
	\begin{gather*}
	6xF_{2n+1}(3x) = B_{2(n+1)}(x)-B_{2n}(x)-(27x^2-4)\sum_{k=0}^{n-1}B_{2(n-k)}(x)F_{2k+1}(3x),\\
	F_{2n}(3x)+F_{2n-2}(3x) = 3xB_{2n-1}(x)-(27x^2-4)\sum_{k=1}^{n-1}B_{2(n-k)-1}(x)F_{2k}(3x),\\
	9x^2F_{2n-1}(3x) = 3xC_{2n-1}(x)-(27x^2-4)\sum_{k=1}^{n-1}C_{2(n-k)-1}(x)F_{2k}(3x),\\
	6xF_{2n}(3x) = xB_{2n}(x)-(27x^2-4)\sum_{k=1}^{n-1}B_{2(n-k)}(x)F_{2k}(3x).\\\notag
	\end{gather*}
\end{corollary}

\section*{4.~Chebyshev-Fibonacci polynomial identities via \break exponential generating functions} 

Functional equations for exponential generating functions will yield connections between 
Chebyshev and Fibonacci polynomials involving binomial coefficients. Recall that we use the following notation
\begin{gather*}
\alpha(x)=x+\sqrt{x^2-1}, \quad \beta(x)=x-\sqrt{x^2-1},\quad \rho(x)=\frac{x+\sqrt{x^2+4}}{2}, \quad \sigma(x)=\frac{x-\sqrt{x^2+4}}{2}.
\end{gather*}
\begin{theorem}
	For $n\geq0$, the following identities hold
	\begin{gather}
	\sum_{k=0}^{n-1}{n\choose k}\bigl(\sqrt{x^2+4}\bigr)^{n-1-k}\bigl(1-(-1)^{n-k}\bigr)T_k(x)\nonumber\\
	= \sum_{k=1}^{n-1}{n\choose k}2^{k-1}\bigl(\sqrt{x^2-1}\bigr)^{n-k}\bigl(1+(-1)^{n-k}\bigr)F_k(x)\label{exp-t-f},\\
\sum_{k=0}^{n-1}{n\choose k}\bigl(\sqrt{x^2+4}\bigr)^{n-1-k}\bigl(1-(-1)^{n-k}\bigr)U_k(x)\nonumber\\
	= \sum_{k=1}^{n-1}{n\choose k}2^{k-1}\bigl(\sqrt{x^2-1}\bigr)^{n-1-k}\bigl(\alpha(x)-(-1)^{n-k}\beta(x)\bigr)F_k(x)\label{exp-u-f}.
	\end{gather}
\end{theorem} 
\begin{proof} 
	To prove formula \eqref{exp-t-f} we use the generating functions \eqref{egf-t} and \eqref{egf-F}. 
	They give the functional equation 
	\begin{equation*}
	2\tau\left(\frac{z}{2},x\right)\sinh\Bigl(\frac{\sqrt{x^2+4}}{2}z\Bigr)=\phi(z,x)\sqrt{x^2+4}\,\cosh\Bigl(\frac{\sqrt{x^2-1}}{2}z\Bigr).
	\end{equation*}
	From this equation we obtain
	\begin{gather*}
	\sum_{n\geq0}\sum_{k=0}^n{n\choose k} \frac{T_k(x)}{2^k}\left(\Bigl(\frac{\sqrt{x^2+4}}{2}\Bigr)^{n-k}
	-\Bigl(-\frac{\sqrt{x^2+4}}{2}\Bigr)^{n-k}\right)\frac{z^n}{n!}\\
	=\frac{\sqrt{x^2+4}}{2}\sum_{n\geq0}\sum_{k=1}^{n}{n\choose k}F_k(x)
	\left(\Bigl(\frac{\sqrt{x^2-1}}{2}\Bigr)^{n-k}+\Bigl(-\frac{\sqrt{x^2-1}}{2}\Bigr)^{n-k}\right)\frac{z^n}{n!}.
	\end{gather*}
	Comparing the coefficients of the both sides gives
	\begin{gather*}
	\sum_{k=0}^n{n\choose k} \frac{T_k(x)}{2^k}\left(\Bigl(\frac{\sqrt{x^2+4}}{2}\Bigr)^{n-k}-\Bigl(-\frac{\sqrt{x^2+4}}{2}\Bigr)^{n-k}\right)\\
	=\frac{\sqrt{x^2+4}}{2}\sum_{k=1}^{n}{n\choose k}F_k(x)
	\left(\Bigl(\frac{\sqrt{x^2-1}}{2}\Bigr)^{n-k}+\Bigl(-\frac{\sqrt{x^2-1}}{2}\Bigr)^{n-k}\right).
	\end{gather*}
	and after simplifications we get \eqref{exp-t-f}. 
	The proof of \eqref{exp-u-f} follows in a similar way and is based on the functional equation 
	\begin{gather*}
	2\sqrt{x^2-1}\sinh\Bigl(\frac{\sqrt{x^2+4}}{2}z\Bigr)\omega\!\left(\frac{z}{2},x\right)\\
	=\sqrt{x^2+4}\left(x\sinh\Bigl(\frac{\sqrt{x^2-1}}{2}z\Bigr)+\sqrt{x^2-1}\cosh\Bigl(\frac{\sqrt{x^2-1}}{2}z\Bigr)\right)\phi(z,x),
	\end{gather*}
	which we derive from generating functions \eqref{egf-u} and \eqref{egf-F}. 
\end{proof}

The next two theorems give us relations involving odd and even indexed Chebyshev and Fibonacci polynomial sequences. 
\begin{theorem}
	For $n\geq0$, the following formulas hold
	\begin{gather*}
	x\sum_{k=0}^{n}{n\choose k}\Bigl(\frac{\sqrt{x^2+4}(2x^3-x)}{x^2+2}\Bigr)^{n-1-k}\bigl(\rho(x)-(-1)^{n-k}\sigma(x)\bigr)T_{2k+1}(x)\nonumber\\ 
	=\bigl(2x\sqrt{x^2-1}\bigl)^{n-1}\sum_{k=0}^{n}{n\choose k}\Bigl(\frac{2x^2-1}{(x^3+2x)\sqrt{x^2-1}}\Bigr)^{k-1}\bigl(\alpha(x)+(-1)^{n-k}\beta(x)\bigr)F_{2k+1}(x) \label{exp-t1-f1}
	\end{gather*}
	and
	\begin{gather*}
	\frac12\sum_{k=0}^{n}{n\choose k}\Bigl(\frac{\sqrt{x^2+4}(2x^3-x)}{x^2+2}\Bigr)^{n-1-k}\bigl(\rho(x)-(-1)^{n-k}\sigma(x)\bigr)U_{2k+1}(x)\nonumber\\ 
	=\bigl(2x\sqrt{x^2-1}\bigr)^{n-2}\sum_{k=0}^{n}{n\choose k}\Bigl(\frac{2x^2-1}{(x^3+2x)\sqrt{x^2-1}}\Bigr)^{k-1}\\
	\hspace{-1.8cm}	\times\bigl(\alpha^2(x)-(-1)^{n-k}\beta^2(x)\bigr)F_{2k+1}(x). \label{exp-u1-f1}
	\end{gather*}
\end{theorem}
\begin{proof} 
	The stated formulas follow from the functional equations
	\begin{gather*}
	\left(x\sinh\Bigl(\frac{(2x^3-x)\sqrt{x^2+4}}{x^2+2}z\Bigr)+\sqrt{x^2+4}\cosh\Bigl(\frac{(2x^3-x)\sqrt{x^2+4}}{x^2+2}z\Bigr)\right)\tau_1(z,x)\\ 
	=\sqrt{x^2+4}\left(x\cosh\bigl(2x\sqrt{x^2-1}z\bigr)+\sqrt{x^2-1}\sinh\bigl(2x\sqrt{x^2-1}z\bigr)\right)\phi_1\left(\frac{4x^2-2}{x^2+2}z,x\right)
	\end{gather*}
	and
	\begin{gather*}
	2\left(x\sinh\Bigl(\frac{(2x^3-x)\sqrt{x^2+4}}{x^2+2}z\Bigr)+\sqrt{x^2+4}\cosh\Bigl(\frac{(2x^3-x)\sqrt{x^2+4}}{x^2+2}z\Bigr)\right)\omega_1(z,x)\\ 
	=\sqrt{\frac{x^2+4}{x^2-1}}\left(\alpha^2(x)e^{2x\sqrt{x^2-1}z}-\beta^2(x)e^{-2x\sqrt{x^2-1}z}\right)\phi_1\left(\frac{4x^2-2}{x^2+2}z,x\right),
	\end{gather*}
	that one can obtain from \eqref{egf-t1}, \eqref{egf-F1} and  \eqref{egf-u1}, \eqref{egf-F1}, respectively.
\end{proof}
\begin{theorem}
	For $n\geq0$, the following formulas hold
	\begin{gather*}
	x\sum_{k=0}^{n-1}{n\choose k}\Bigl(\frac{\sqrt{x^2+4}(2x^3-x)}{x^2+2}\Bigr)^{n-k-1}\bigl(1-(-1)^{n-k}\bigr)T_{2k}(x)\nonumber\\ 
	=\bigl(2x\sqrt{x^2-1}\bigr)^{n-1}\sum_{k=1}^{n}{n\choose k}\Bigl(\frac{2x^2-1}{(x^3+2x)\sqrt{x^2-1}}\Bigr)^{k-1}\bigl(1+(-1)^{n-k}\bigr)F_{2k}(x)
	\end{gather*}
	and  
	\begin{gather*}
	\frac12\sum_{k=0}^{n}{n\choose k}\Bigl(\frac{\sqrt{x^2+4}(2x^3-x)}{x^2+2}\Bigr)^{n-k-1}\bigl(1-(-1)^{n-k}\bigr)U_{2k}(x)\nonumber\\ 
	=\bigl(2x\sqrt{x^2-1}\bigr)^{n-2}\sum_{k=1}^{n}{n\choose k}\Bigl(\frac{2x^2-1}{(x^3+2x)\sqrt{x^2-1}}\Bigr)^{k-1}\bigl(\alpha(x)-(-1)^{n-k}\beta(x)\bigr)F_{2k}(x).
	\end{gather*}
\end{theorem}
\begin{proof} Generating functions \eqref{egf-t2}, \eqref{egf-F2} and \eqref{egf-u1}, \eqref{egf-F2}, respectively,  yield 
	\begin{gather*}
	2\sinh\Bigl(\frac{(2x^3-x)\sqrt{x^2+4}}{x^2+2}z\Bigr)\tau_2(z,x)=\sqrt{x^2+4}\cosh\bigl(2x\sqrt{x^2-1}z\bigr)\phi_2\left(\frac{4x^2-2}{x^2+2}z,x\right)
	\end{gather*}
	and
	\begin{gather*}
	\sqrt{x^2-1}\sinh\Bigl(\frac{(2x^3-x)\sqrt{x^2+4}}{x^2+2}z\Bigr)\omega_2(z,x)\\ 
	=2\sqrt{x^2+4}\left(x\sinh\bigl(2x\sqrt{x^2-1}z\bigr)+\sqrt{x^2-1}\cosh\bigl(2x\sqrt{x^2-1}z\bigr)\right)\phi_2\left(\frac{4x^2-2}{x^2+2}z,x\right).
	\end{gather*}
	The results follow from writing in terms of power series and collecting terms.
\end{proof}

The last theorem contains additional relations for Chebyshev and Fibonacci polynomials that we found. 
\begin{theorem}
	For $n\geq0$, the following formulas hold
	\begin{gather*}
	\sum_{k=0}^{n-1}{n\choose k}\Bigl(\frac{\sqrt{x^2+4}(2x^2-1)}{x}\Bigr)^{n-k-1}\bigl(1-(-1)^{n-k}\bigr)T_{2k}(x)\nonumber\\ 
	=\sum_{k=1}^{n}{n\choose k}\Bigl(\frac{4x^2-2}{x}\Bigr)^{k-1}\bigl(2x\sqrt{x^2-1}\bigr)^{n-k}\bigl(1+(-1)^{n-k}\bigr)F_{k}(x),\\[2pt]
	\sum_{k=0}^{n}{n\choose k}\Bigl(\frac{\sqrt{x^2+4}(2x^2-1)}{x}\Bigr)^{n-k-1}\bigl(1-(-1)^{n-k}\bigr)U_{2k}(x)\nonumber\\ 
	=2x\sum_{k=1}^{n}{n\choose k}\Bigl(\frac{4x^2-2}{x}\Bigr)^{k-1}\bigl(2x\sqrt{x^2-1}\bigr)^{n-k-1}\bigl(\alpha(x)-(-1)^{n-k}\beta(x)\bigr)F_{k}(x),\\[2pt]
	\sum_{k=0}^{n}{n\choose k}\Bigl(\frac{\sqrt{x^2+4}(2x^2-1)}{x}\Bigr)^{n-k-1}\bigl(1-(-1)^{n-k}\bigr)T_{2k+1}(x)\nonumber\\ 
	=\sum_{k=1}^{n}{n\choose k}\Bigl(\frac{4x^2-2}{x}\Bigr)^{k-1}\bigl(2x\sqrt{x^2-1}\bigr)^{n-k}\bigl(\alpha(x)+(-1)^{n-k}\beta(x)\bigr)F_{k}(x),\\[2pt]
	\sum_{k=0}^{n}{n\choose k}\Bigl(\frac{\sqrt{x^2+4}(2x^2-1)}{x}\Bigr)^{n-k-1}\bigl(1-(-1)^{n-k}\bigr)U_{2k+1}(x)\nonumber\\ 
	=x\sum_{k=1}^{n}{n\choose k} \Bigl(\frac{4x^2-2}{x}\Bigr)^{k-1}\bigl(2x\sqrt{x^2-1}\bigr)^{n-k-1}\bigl(\alpha^2(x)-(-1)^{n-k}\beta^2(x)\bigr)F_{k}(x),\\[2pt]
		x\sum_{k=0}^{n}{n\choose k}\Bigl(\frac{x^2+2}{2x}\Bigr)^k \Bigl(\frac{x\sqrt{x^2+4}}{2}\Bigr)^{n-k-1}\bigl(\rho(x)-(-1)^{n-k}\sigma(x)\bigr)T_{k}(x)\nonumber\\ 
	=\sum_{k=0}^{n}{n\choose k}\Bigl(\frac{(x^2+2)\sqrt{x^2-1}}{2x}\Bigr)^{n-k}\bigl(1+(-1)^{n-k}\bigr)F_{2k+1}(x),\\[2pt]
		x\sum_{k=0}^{n}{n\choose k}\Bigl(\frac{x^2+2}{2x}\Bigr)^{k-1}\Bigl(\frac{x\sqrt{x^2+4}}{x}\Bigr)^{n-k-1}\bigl(\rho(x)-(-1)^{n-k}\sigma(x)\bigr)U_{k}(x)\nonumber\\ 
	=\sum_{k=0}^{n}{n\choose k} \Bigl(\frac{(x^2+2)\sqrt{x^2-1}}{2x}\Bigr)^{n-k-1}\bigl(\alpha(x)-(-1)^{n-k}\beta(x)\bigr)F_{2k+1}(x),\\[2pt]
		\sum_{k=0}^{n}{n\choose k}\Bigl(\frac{x^2+2}{2x}\Bigr)^k\Bigl(\frac{x\sqrt{x^2+4}}{2}\Bigr)^{n-k-1}\bigl(1-(-1)^{n-k}\bigr)T_{k}(x)\nonumber\\ 
	=\sum_{k=1}^{n}{n\choose k} \Bigl(\frac{(x^2+2)\sqrt{x^2-1}}{2x}\Bigr)^{n-k}\bigl(1+(-1)^{n-k}\bigr)F_{2k}(x),
	\end{gather*}
	\begin{gather*}
	x\sum_{k=0}^{n}{n\choose k} \Bigl(\frac{x^2+2}{2x}\Bigr)^{k-1}\Bigl(\frac{x\sqrt{x^2+4}}{2}\Bigr)^{n-k-1}\bigl(1-(-1)^{n-k}\bigr)U_{k}(x)\nonumber\\
	=\sum_{k=0}^{n}{n\choose k} \Bigl(\frac{(x^2+2)\sqrt{x^2-1}}{2x}\Bigr)^{n-k-1} \bigl(\alpha(x)+(-1)^{n-k}\beta(x)\bigr)F_{2k}(x),\nonumber\\[2pt]
	x\sum_{k=0}^{n}{n\choose k} \Bigl(\frac{x^2+2}{4x^2-2}\Bigr)^{k}\Bigl(\frac{x\sqrt{x^2+4}}{2}\Bigr)^{n-k}\bigl(1-(-1)^{n-k}\bigr)T_{2k+1}(x)\nonumber\\ 
	=\sum_{k=0}^{n}{n\choose k}\Bigl(\frac{x(x^2+2)\sqrt{x^2-1}}{2x^2-1}\Bigr)^{n-k} \bigl(\alpha(x)+(-1)^{n-k}\beta(x)\bigr)F_{2k}(x),\nonumber\\[2pt]
	\sum_{k=0}^{n}{n\choose k} \Bigl(\frac{x^2+2}{4x^2-2}\Bigr)^{k-1}\Bigl(\frac{x\sqrt{x^2+4}}{2}\Bigr)^{n-k-1}\bigl(1-(-1)^{n-k}\bigr)U_{2k+1}(x)\nonumber\\ 
	=2\sum_{k=0}^{n}{n\choose k} \Bigl(\frac{x(x^2+2)\sqrt{x^2-1}}{2x^2-1}\Bigr)^{n-k-1} \bigl(\alpha^2(x)-(-1)^{n-k}\beta^2(x)\bigr)F_{2k}(x),\nonumber\\[2pt]
	x\sum_{k=0}^{n}{n\choose k}\Bigl(\frac{x^2+2}{4x^2-2}\Bigr)^{k}\Bigl(\frac{x\sqrt{x^2+4}}{2}\Bigr)^{n-k-1}\bigl(\rho(x)-(-1)^{n-k}\sigma(x)\bigr)T_{2k}(x)\nonumber\\ 
	=\sum_{k=0}^{n}{n\choose k} \Bigl(\frac{x(x^2+2)\sqrt{x^2-1}}{2x^2-1}\Bigr)^{n-k} \bigl(1+(-1)^{n-k}\bigr)F_{2k+1}(x),\nonumber\\[2pt]
	\sum_{k=0}^{n}{n\choose k} \Bigl(\frac{x^2+2}{4x^2-2}\Bigr)^{k-1}\Bigl(\frac{x\sqrt{x^2+4}}{2}\Bigr)^{n-k-1}\bigl(\rho(x)-(-1)^{n-k}\sigma(x)\bigr)U_{2k}(x)\nonumber\\ 
	=2\sum_{k=0}^{n}{n\choose k} \Bigl(\frac{x(x^2+2)\sqrt{x^2-1}}{2x^2-1}\Bigr)^{n-k-1} \bigl(\alpha(x)-(-1)^{n-k}\beta(x)\bigr)F_{2k+1}(x).\nonumber\\\notag
	\end{gather*}
\end{theorem}

\section*{5.~Concluding comments: Fibonacci and Lucas identities implied by Chebyshev-Fibonacci identities}

The polynomial relations derived in this paper imply many Fibonacci and Lucas identities, 
some of which are certainly known but some of which could turn out to be new. 
These identities come from the various links between Chebyshev polynomials and Fibonacci (Lucas) numbers. 
In \cite{Cas} and \cite{Siyi} many such links are listed. 
Among the various connections we have
\begin{gather}
T_n\Bigl(\frac{3}{2}\Bigr) = \frac{1}{2}L_{2n}, \qquad U_n\Bigl(\frac{3}{2}\Bigr) = F_{2n+2}, \label{rel1}\\[3pt]
T_n\Bigl(\frac{i}{2}\Bigr) = \frac{i^n}{2}L_{n}, \qquad U_n\Bigl(\frac{i}{2}\Bigr) = i^nF_{n+1}, \label{rel2} \notag
\end{gather}
\begin{gather}
T_{2n}\Bigl(\frac{\sqrt{5}}{2}\Bigr) = \frac{1}{2}L_{2n}, \qquad U_{2n}\Bigl(\frac{\sqrt{5}}{2}\Bigr) = L_{2n+1}, \label{rel3}\notag\\[3pt]
T_{2n+1}\Bigl(\frac{\sqrt{5}}{2}\Bigr) = \frac{\sqrt{5}}{2}F_{2n+1}, \qquad U_{2n+1}\Bigl(\frac{\sqrt{5}}{2}\Bigr) = \sqrt{5} F_{2n+2}. 
\label{rel4}\notag
\end{gather}

Using \eqref{rel1}, for instance, from Theorems \ref{Theo1}--\ref{Theo3}, we can immediately obtain new families of Fibonacci and Lucas identities. In the next statement, we state some examples.
\begin{corollary} 
	For $n\geq1$, we have the following identities:
	\begin{gather*}
	15F_{2(2n-1)}=5\cdot4^n-20\cdot4^{-n}+11\sum_{k=1}^{n-1}\bigl(4^{k}-4^{-k}\bigr)F_{2(2n-2k-1)},\\
	15F_{4n}=12\bigl(4^n-4^{-n}\bigr)+11\sum_{k=1}^{n-1}\bigl(4^{k}-4^{-k}\bigr)
	F_{4(n-k)},\\
	15F_{2n}=4\bigl(4^n-4^{-n}\bigr)-5\sum_{k=1}^{n-1}\bigl(4^{k}-4^{-k}\bigr)F_{2(n-k)},\\
	15L_{4(n-1)}=4^n+104\cdot4^{-n}+11\sum_{k=1}^{n-1}\bigl(4^{k}-4^{-k}\bigr)L_{4(n-k-1)},\\
	5L_{2(2n-1)}=9\cdot4^n+36\cdot4^{-n}+11\sum_{k=1}^{n-1}\bigl(4^{k}-4^{-k}\bigr)
	L_{2(2n-2k-1)},\\
	3L_{2n-2}=4^n+8\cdot4^{-n}-\sum_{k=0}^{n-1}\bigl(4^{k}-4^{-k}\bigr)L_{2(n-k-1)}.
	\end{gather*}
\end{corollary}

To give another example, observe that from
\begin{equation*}
T_n(-\sqrt{5}) = \left.
\begin{cases}
\frac{1}{2} L_{3n}, & \mbox{$n$ even,} \\
-\frac{\sqrt{5}}{2} F_{3n}, & \mbox{$n$ odd},
\end{cases}
\right. \qquad
U_n(-\sqrt{5}) = \left.
\begin{cases}
\frac{1}{4} L_{3n+3}, & \mbox{$n$ even,} \\
-\frac{\sqrt{5}}{4} F_{3n+3}, & \mbox{$n$ odd},
\end{cases}
\right.
\end{equation*}
and
\begin{displaymath}
F_n(-\sqrt{5}) = \left.
\begin{cases}
-\frac{\sqrt{5}}{3} F_{2n}, & \mbox{$n$ even,} \\
\frac{1}{3} L_{2n}, & \mbox{$n$ odd},
\end{cases}
\right. 
\end{displaymath}
from Theorem \ref{Theo2} we get the next summation identities.
\begin{corollary} 
	Let $n\geq 0$. Then
	\begin{gather*}
	11 \sum_{k=1}^{n} F_{4k} L_{6(n-k)} = 3 L_{6n} - 2F_{4n+4} + 18 F_{4n}, \\
	11 \sum_{k=1}^{n} F_{4k} L_{6(n-k)+3} = 3 L_{6n+3} - 4F_{4n+4} - 4 F_{4n}, 
	\end{gather*}
	\begin{gather*}
	11 \sum_{k=0}^{n} L_{4k+2} F_{6(n-k)+3} = 3 (F_{6n+9} - F_{6n+3}) - 2(L_{4n+6}-L_{4n+2}), \\
	11 \sum_{k=0}^{n} L_{4k+2} F_{6(n-k)+6} = 3 (F_{6n+12} - F_{6n+6}) - 8L_{4n+6}.
	\end{gather*}
\end{corollary}

Finally, with $F_n(4)=F_{3n}/2$, we get from Theorem \ref{Theo1} the following. 
\begin{corollary} 
	Let $n\geq 1$. Then
	\begin{gather*}
	F_{3n} = 2T_{n-1}(4) - \sum_{k=1}^{n-2} \bigl(4T_{n-1-k}(4)-2T_{n-2-k}(4)\bigr)F_{3k}, \\
	F_{3n} + 4 F_{3n-3} = 2U_{n-1}(4) -\sum_{k=1}^{n-2}\bigl(4U_{n-1-k}(4)-2U_{n-2-k}(4)\bigr)F_{3k},
	\end{gather*}
	with
	\begin{equation*}
	T_{n}(4) = 4^n \sum_{j=0}^{\lfloor{n/2}\rfloor} {n \choose 2j}\Big (\frac{15}{16}\Big )^j \quad \mbox{and} \quad 
	U_{n}(4) = 4^n \sum_{j=0}^{\lfloor{n/2}\rfloor} {n+1 \choose 2j+1}\Big (\frac{15}{16} \Big )^j.
	\end{equation*}
\end{corollary}

More experiments in this direction are left for a personal study. 

\vspace{0,4cm}

\bigskip
\hrule
\bigskip

\noindent 2010 {\it Mathematics Subject Classification}: 11B37, 11B39.

\bigskip
\noindent \emph{Keywords:} Chebyshev polynomials,  Fibonacci polynomials, Fibonacci numbers,  Lucas numbers, balancing polynomials,  generating function.

\end{document}